\documentclass[reqno]{amsart}
\usepackage{amsmath,amssymb,amsthm,graphicx,a4wide}
\usepackage[small,bf]{caption}
\usepackage{hyperref}
\usepackage{enumerate}
\usepackage{tikz}
\theoremstyle{plain}
\newtheorem{thm}{Theorem}
\newtheorem{lemma}[thm]{Lemma}

\theoremstyle{definition}

\theoremstyle{remark}
\newtheorem{rem}{Remark}

\newcommand{\prn}[1]{\left(#1\right)}
\newcommand{\abs}[1]{\left|#1\right|}
\newcommand{\brk}[1]{\left[#1\right]}

\newcommand{\ud}[1]{{\text{d}#1}}

\begin{document}
\parskip1ex

\title[An exact particle method and its application to stiff reaction kinetics]
{An exact particle method for scalar conservation laws and its application
to stiff reaction kinetics}
\author{Yossi Farjoun}
\address[Yossi Farjoun]
{G.\ Mill\'an Institute of Fluid Dynamics\\
Nanoscience and Industrial Mathematics\\
University Carlos III de Madrid\\
Av.\ Universidad 30, 28911 Legan\'es\\
Spain}
\email{yfarjoun@ing.uc3m.es}
\author{Benjamin Seibold}
\address[Benjamin Seibold]
{Department of Mathematics \\ Temple University \\
1801 North Broad Street \\ Philadelphia, PA 19122}
\email{seibold@temple.edu}
\urladdr{http://www.math.temple.edu/\~{}seibold}
\date{\today}
\begin{abstract}
An ``exact'' method for scalar one-dimensional hyperbolic conservation laws is presented. The approach is based on the evolution of shock particles, separated by local similarity solutions. The numerical solution is defined everywhere, and is as accurate as the applied ODE solver. Furthermore, the method is extended to stiff balance laws. A special correction approach yields a method that evolves detonation waves at correct velocities, without resolving their internal dynamics. The particle approach is compared to a classical finite volume method in terms of numerical accuracy, both for conservation laws and for an application in reaction kinetics.
\end{abstract}

\maketitle

\section{Introduction}
In this paper, a special class of numerical methods for scalar hyperbolic conservation laws in one space dimension is presented. An important area in which such problems arise is the simulation of nonlinear flows in networks. An example is the flow of vehicular traffic on highways \cite{HertyKlar2003}. The flow along each network edge is described by a hyperbolic conservation law (e.g.~the Lighthill-Whitham model \cite{LighthillWhitham1955} for traffic flow). The edges meet at the network vertices, where problem specific coupling conditions are imposed (such as the Coclite-Piccoli conditions \cite{HoldenRisebro1995, CocliteGaravelloPiccoli2005} for traffic flow). Here, we focus on the evolution of the flow along a single edge. In network flows, high requirements are imposed on the numerical method. On the one hand, the approach must guarantee exact conservation (no cars must be lost), no spurious oscillations must occur (otherwise one may encounter negative densities), and shocks (traffic jams) should be located accurately. On the other hand, in the simulation of a large network, only very few computational resources can be attributed to each edge.

A commonly used approach is to approximate the governing conservation law by traditional finite difference \cite{LaxWendroff1960} or finite volume methods \cite{Godunov1959}. Low order methods are generally too diffusive, thus do not admit an accurate location of shocks. High order methods, such as finite volume methods with limiters \cite{VanLeer1974}, or ENO \cite{HartenEngquistOsherChakravarthy1987}/WENO \cite{LiuOsherChan1994} schemes admit more accurate capturing of shocks. However, this comes at the expense of locality: stencils reach over multiple cells, which poses challenges at the network vertices. Alternative approaches are front tracking methods \cite{HoldenHoldenHeghKrohn1988}. These do not operate on a fixed grid, but track shocks explicitly. Thus, shock are located accurately. However, smooth parts, such as rarefaction fans, are not represented very well. Another class of approaches is based on the underlying method of characteristics. An example is the CIR method \cite{CourantIsaacsonRees1952}, which updates information on grid points by tracing characteristic curves. Thus, it is a fixed-grid finite difference method. A fully characteristic approach was presented by Bukiet et al. \cite{BukietPeleskoLiSachdev1996} that tracks the evolution of particles. Where the solution is smooth, particles follow the characteristic curves, and where these curves collide, shocks are evolved. By construction, shocks are ideally sharp. While the tracing of the characteristics is high order accurate, the location of shocks is only first order accurate.

Another approach was presented by the authors \cite{FarjounSeibold2009_1, FarjounSeibold2009_2, FarjounSeibold2009_3}. In contrast to previous methods, here shocks are resolved by the merging of characteristic particles. This is made possible by the definition of a suitable interpolation, which is a similarity solution of the underlying conservation law. Hence, in \cite{FarjounSeibold2009_3} we suggest to call the approach \emph{rarefaction tracking}. The method admits second order accurate location of shocks. In Sect.~\ref{sec:characteristic_particles_interpolation}, the fundamentals of this approach, in particular the similarity interpolation, are outlined.

A generalization of the characteristic particle method presented in \cite{FarjounSeibold2009_2}, namely \emph{shock particles}, is introduced in Sect.~\ref{sec:shock_particles}. A shock particle is a moving discontinuity that carries two function values. If the jump height is zero, a classical characteristic particle is recovered. Even though shocks are evolved explicitly, the aforementioned similarity interpolation still plays a crucial role, thus the approach is fundamentally different from traditional shock/front tracking methods. In fact, the presented approach solves the considered hyperbolic conservation law exactly, up to the integration error of an ODE. Hence, we call it an \emph{exact particle method}. In Sect.~\ref{sec:exact_particle_method}, the evolution and interaction of shock particles is shown to give rise to an actual computational method. The key idea is that the original partial differential equation is reduced to an ordinary differential equation, which then can be solved using a high order ODE solver. A comparison of the exact particle approach with a traditional finite volume method (using the package CLAWPACK \cite{Clawpack}) is presented in Sect.~\ref{sec:numerical_error_analysis_particle_method}.

The presented approach is highly accurate for hyperbolic conservation laws, and is quite amenable to extension to balance laws. Of particular interest here are stiff reaction kinetics, in which reactions happen on a faster time scale than the nonlinear advection. In Sect.~\ref{sec:reaction_kinetics}, the problem is introduced and some properties of its solution are described. In Sect.~\ref{sec:particl_method_stiff_reaction_kinetics}, a specialized adaptation of the particle method is presented. It is based on the exact particle method introduced in Sect.~\ref{sec:exact_particle_method}, with a fundamental adaptation that uses the similarity interpolation to provide a certain level of subgrid resolution near the detonation wave. This method is able to track detonation waves correctly, without resolving them explicitly. Computational results for this application are presented in Sect.~\ref{sec:numerical_results_reaction_kinetics}.

\section{Characteristic Particles and Similarity Solution Interpolant}
\label{sec:characteristic_particles_interpolation}
Consider a scalar conservation law in one space dimension
\begin{equation}
u_t+\prn{f(u)}_x = 0\;,
\quad u(x,0) = u_0(x)\;.
\label{eq:conservation_law}
\end{equation}
The flux function $f$ is assumed to be twice differentiable and either convex ($f''>0$) or concave ($f''<0$) on the range of function values. We consider an approximation to the true solution of \eqref{eq:conservation_law} by a family of functions, defined as follows.

Consider a finite number of particles. A particle is a computational node that carries a (variable) position $x_i$, and a (variable) function value $u_i$. Let the set of particles be defined by $P = \{(x_1,u_1),\dots,(x_n,u_n)\}$ with $x_1\le\dots\le x_n$. On the interval $[x_1,x_n]$, we define the interpolant $U_P(x)$ piecewise on the intervals between neighboring particles, as follows. If $u_i = u_{i+1}$, then on the interpolant on $[x_i,x_{i+1}]$ is constant $U_P(x) = u_i$. Otherwise the interpolant on $[x_i,x_{i+1}]$ satisfies
\begin{equation}
\frac{x-x_i}{x_{i+1}-x_i} = \frac{f'(U_P(x))-f'(u_i)}{f'(u_{i+1})-f'(u_i)}\;.
\label{eq:interpolation}
\end{equation}
This defines the inverse interpolant $x(U_P)$ explicitly on $[x_i,x_{i+1}]$. Since $f$ is convex or concave, the interpolant $U_P(x)$ itself is uniquely defined. As shown in \cite{FarjounSeibold2009_2}, the interpolation $U_P$ defined above is an analytical solution of the conservation law \eqref{eq:conservation_law}, in the following sense.
Consider particles moving ``sideways'' according to $P(t) = \{(x_1+f'(u_1)t,u_1),\dots,(x_n+f'(u_n)t,u_n)\}$. If at time $t = 0$, the particles satisfy $x_1<\dots<x_n$, then for sufficiently short times $t>0$, the interpolant $U_{P(t)}$, defined by \eqref{eq:interpolation} is \emph{the} analytical solution to the conservation law \eqref{eq:conservation_law} at time $t$, starting with initial conditions $u_0 = U_P(0)$. This follows from the fact that each point $(x(t),u(t))$ on the solution moves according to the characteristic equations of \eqref{eq:conservation_law}, which are
\begin{equation}
\left\{
\begin{aligned}
\dot x &= f'(u) \\
\dot u &= 0\;.
\end{aligned}
\right.
\label{eq:characteristic_equations}
\end{equation}
From the definition of $P(t)$ it is obvious that the particles satisfy \eqref{eq:characteristic_equations}. Any other point is given by the above defined interpolation. Replacing $x_i$ by $x_i+f'(u_i)t$ in \eqref{eq:interpolation} and differentiating with respect to $t$ yields that $\dot x(t) = f'(U_P(x))$.

The solution between neighboring particles is a similarity solution that either comes from a discontinuity (if the particles depart) or becomes a shock (if the particles approach each other). Hence, the solution $U_P(x)$ is composed of rarefaction waves and compression waves. Therefore, as described in \cite{FarjounSeibold2009_3}, the approach can be interpreted as ``rarefaction tracking'', which expresses its similarity and fundamental difference to front tracking approaches \cite{HoldenHoldenHeghKrohn1988,HoldenRisebro2002}.

\begin{figure}
\centering
\begin{minipage}[t]{.7\textwidth}
\includegraphics[width=\textwidth]{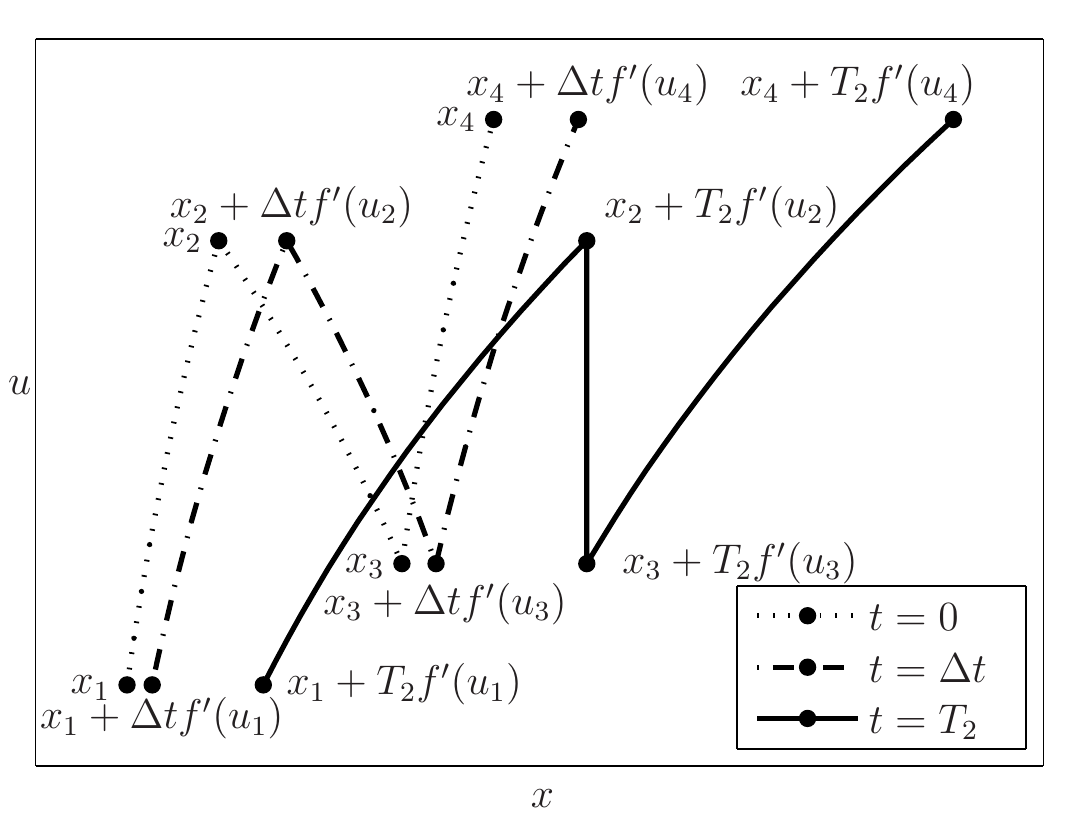}
\end{minipage}
\caption{An illustration of characteristic particles moving according to \eqref{eq:characteristic_equations}. The solution develops a shock at $t = T_2$. The dotted line is the interpolation between the particles at time $t$, the dash-dotted line--after a short time $\Delta t$. The solid line shows the solution at $t = T_2$ when the solution develops a shock between particles 2 and 3.}
\label{fig:characteristic_particles}
\end{figure}

The interpolant $U_P$ \eqref{eq:interpolation} is a solution of \eqref{eq:conservation_law} until the time of the first collision, i.e.~the moment when two neighboring particles share the same $x$-position. For a pair of neighboring particles $(x_i,u_i)$ and $(x_{i+1},u_{i+1})$, the time of collision is given by
\begin{equation}
T_i = -\frac{x_{i+1}-x_i}{f'(u_{i+1})-f'(u_i)}\;.
\label{eq:collision_time}
\end{equation}
If particles depart from each other (i.e.~$f'(u_i)<f'(u_{i+1})$), one has $T_i<0$, thus no collision happens in future time. For a set of $n$ particles, the first time of collision is $T^* = \min\prn{\{T_i:T_i>0\}\cup\infty}$. The solution is continuous until that time. At $t = T^*$, a shock occurs (at $x_i$, between $u_{i+1}$ and $u_i$), and the method of characteristics alone does not yield a correct solution further in time. An illustration of the particle movement and the development of a shock can be seen in Fig.~\ref{fig:characteristic_particles}.
\begin{rem}
We assume that one is interested in single-valued weak entropy-solutions, which possess shocks. In some applications multi-valued solutions are sought, and those can be obtained easily by continuing to move the particles according to the characteristic equations \eqref{eq:characteristic_equations} beyond the occurrence of shocks.
\end{rem}
As presented in previous papers \cite{FarjounSeibold2009_1,FarjounSeibold2009_2,FarjounSeibold2009_3}, the use of the method of characteristics even in the presence of shocks can be made possible by a suitably designed particle management. Particles that collide are immediately merged into a single particle, with a new function value $u$ chosen such that the total area under the interpolant $U_P$ is preserved. After this merge, the defined interpolant is again continuous, and one can step further in time using solely characteristic particle movement \eqref{eq:characteristic_equations}. This approach introduces a small error around shocks. Suitable insertion of new particles near shocks (before merging) guarantees that the error remains localized near shocks.

\section{Shock Particles}
\label{sec:shock_particles}
In this paper, we present an approach that does not introduce any errors intrinsically. While the approximation of general initial conditions by finitely many particles involves an error (see Sect.~\ref{subsec:approximation_initial_conditions}), the actual evolution under the conservation law \eqref{eq:conservation_law} is exact---not just pointwise on the particles, but in the sense of functions. The new approach generalizes characteristic particles to \emph{shock particles}.

\subsection{Evolution of Shock Particles}
A shock particle is a computational node that carries a (variable) position $x_i$, a (variable) left state $u_i^-$, and a (variable) right state $u_i^+$, which satisfy the Oleinik entropy condition \cite{Evans1998} $f'(u_i^-)\ge f'(u_i^+)$). Whenever a shock particle $(x_i,u_i^-,u_i^+)$ violates this conditions (e.g.~because it is placed in the initial conditions), it is immediately replaced by two particles $(x_i,u_i^-,u_i^-)$ and $(x_i,u_i^+,u_i^+)$, which then depart from each other (since $f'(u_i^-)<f'(u_i^+)$). Given $n$ shock particles $P = \{(x_1,u_1^-,u_1^+),\dots,(x_n,u_n^-,u_n^+)\}$, the interpolant $U_P$ on $[x_1,x_n]$ is defined piecewise: on $[x_i,x_{i+1}]$, it satisfies
\begin{equation}
\frac{x-x_i}{x_{i+1}-x_i} = \frac{f'(U_P(x))-f'(u_i^+)}{f'(u_{i+1}^-)-f'(u_i^+)}\;.
\label{eq:interpolation_shock}
\end{equation}
The velocity of a shock particle is given by the Rankine-Hugoniot condition \cite{Evans1998}
as $\dot x_i = s(u_i^-,u_i^+)$, where
\begin{equation}
s(u,v) = \begin{cases}
\frac{f(u)-f(v)}{u-v} & u\neq v \\
\quad f'(u)           & u=v
\end{cases}
\label{eq:difference_quotient}
\end{equation}
is the difference quotient of $f$, continuously extended at
$u=v$.

\begin{rem} When implementing this function numerically, one should avoid
calculating the difference quotient when the distance $\abs{u-v}$ is
very small. For those cases one should consider the limiting value
$f'(u)$ as a more accurate alternative, or even better, the next order
Taylor expansion,
\begin{equation}
s(u,u+\epsilon) \approx f'(u) + \tfrac12\epsilon f''(u),
\end{equation}
can be used.
\end{rem}

At a shock, the function values $u_i^-$ and $u_i^+$ change in time as well. Their rate of change is exactly such that the interpolation \eqref{eq:interpolation_shock} near the shock evolves as a smooth function should evolve under \eqref{eq:conservation_law}. Here we derive the evolution for the right value of the shock $u_i^+$. The argument for $u_i^-$ works analogously. If we had $u_i^- = u_i^+$, the particle would have a function value that is constant in time $\dot u_i^+=0$, and move with velocity $\dot x_i = f'(u_i^+)$. The interpolant \eqref{eq:interpolation_shock} with these definitions for $\dot x_i$ and $\dot u_i^+$ is the correct solution between $x_i$ and $x_{i+1}$. If $u_i^- \neq u_i^+$, the shock moves at a speed different from $f'(u_i^+)$. In order to preserve the same \emph{interpolation}, the function value $u_i^+$ has to evolve according to
\begin{equation}
\dot u_i^+ = (\dot x_i-f'(u_i^+))\frac{f'(u_{i+1}^-)-f'(u_i^+)}{x_{i+1}-x_i}
\frac{1}{f''(u_i^+)}\;.
\end{equation}
Here $\dot x_i-f'(u_i^+)$ is the relative velocity of the shock to a characteristic particle velocity, and $\frac{f'(u_{i+1}^-)-f'(u_i^+)}{x_{i+1}-x_i}\frac{1}{f''(u_i^+)}$ is the slope of the interpolant \eqref{eq:interpolation_shock} at $x_i$, found by differentiating \eqref{eq:interpolation_shock} with respect to $x$.
The law of motion for a shock particle is thus
\begin{equation}
\left\{
\begin{aligned}
\dot x_i   &= s(u_i^-,u_i^+) \\
\dot u_i^- &= \prn{s(u_i^-,u_i^+)-f'(u_i^-)}
\frac{f'(u_{i-1}^+)-f'(u_i^-)}{x_{i-1}-x_i}\frac{1}{f''(u_i^-)} \\
\dot u_i^+ &= \prn{s(u_i^-,u_i^+)-f'(u_i^+)}
\frac{f'(u_{i+1}^-)-f'(u_i^+)}{x_{i+1}-x_i}\frac{1}{f''(u_i^+)}
\end{aligned}
\right.
\label{eq:motion_shock_particle}
\end{equation}
Observe that the evolution of a shock particle depends on the neighboring two particles. See Fig.~\ref{fig:shock_particles} for an illustration of the derivation of these equations.

In the case $u_i^- = u_i^+$, we call a shock particle \emph{characteristic}. In fact, a characteristic particle, as described in Sect.~\ref{sec:characteristic_particles_interpolation}, is nothing else than a shock particle with jump height zero. This is motivated by
\begin{lemma}
The motion of a shock particle \eqref{eq:motion_shock_particle} reduces to the motion of a characteristic particle \eqref{eq:characteristic_equations}, as $u_i^+ - u_i^- \to 0$.
\end{lemma}
\begin{proof}
By definition \eqref{eq:difference_quotient}, the first equation in \eqref{eq:motion_shock_particle} clearly reduces to $\dot x_i= f'(u_i)$. In the second and third equation, the fraction remains finite while the quantity in the parentheses converges to zero, and thus the whole expression vanishes $\dot u_i^-\to 0$, and $\dot u_i^+\to 0$.
\end{proof}
\begin{rem}
In a numerical implementation, the right hand side \eqref{eq:motion_shock_particle} can almost be implemented as it stands, with the only modification that the difference quotient \eqref{eq:difference_quotient} is replaced by the characteristic speed if $\abs{u_i^+ - u_i^-}$ is less than a sufficiently small value.
\end{rem}
\begin{thm}
\label{thm:exact_solution}
If the time evolution \eqref{eq:motion_shock_particle} of the particles $P(t)$ is solved exactly, then for sufficiently short times $t$, the resulting evolution of the interpolation $U_{P(t)}$ is the unique weak entropy solution of the conservation law \eqref{eq:conservation_law} with initial conditions $u_0 = U_{P(0)}$, on the domain of definition $[x_0(t),x_n(t)]$.
\end{thm}
\begin{proof}
Due to the first equation in \eqref{eq:motion_shock_particle}, all shocks (including those of height zero) move at their correct speeds. By construction, every shock satisfies the entropy condition. Discontinuities that violate the entropy condition immediately become rarefaction waves. The second and third equation in \eqref{eq:motion_shock_particle} ensure that each point on the interpolation between particles moves according to the characteristic equations \eqref{eq:characteristic_equations}.
\end{proof}

\begin{figure}
\centering
\begin{minipage}[t]{.9\textwidth}
\includegraphics[width=\textwidth]{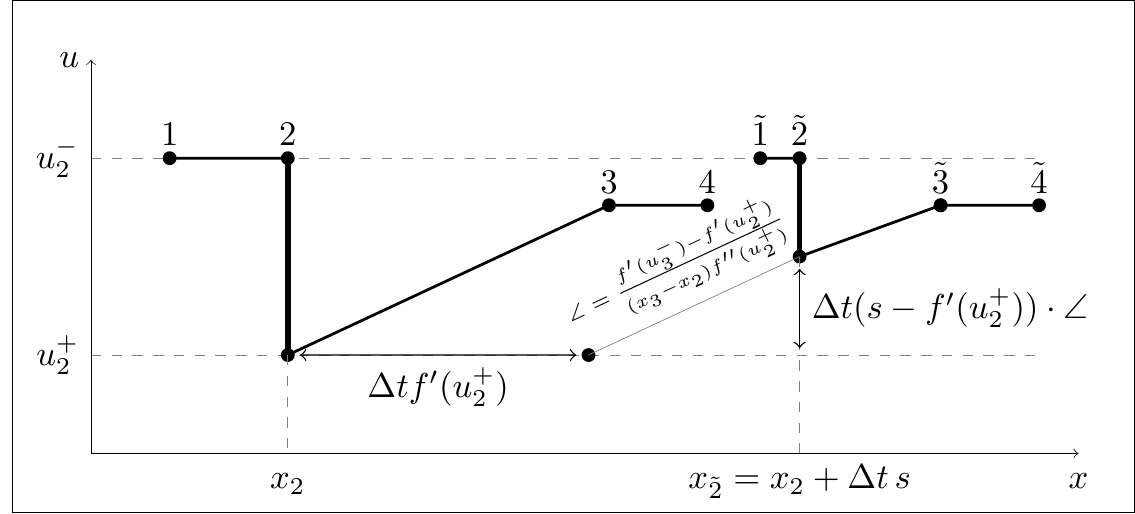}
\end{minipage}
\caption{The shock moves with speed $s$, given by the Rankine-Hugoniot
  condition, which is different than the characteristic speed. 
Thus shock particles must have varying function values.
 In the figure $ \tilde{\cdot} $ denotes moved particles, and
 $\angle$ is the slope of the similarity solution to the left of
 particle 2.}
\label{fig:shock_particles}
\end{figure}

\subsection{Interaction of Shock Particles}
After some time, neighboring shock particles may collide, i.e.~share the same $x$-position. In this situation, the two shocks become a single shock, as the following Lemma shows.
\begin{lemma}
Two neighboring shock particles $(x_i,u_i^-,u_i^+)$ and $(x_{i+1},u_{i+1}^-,u_{i+1}^+)$ satisfy $u_i^+ = u_{i+1}^-$ at their time of collision, if at least one of them is not characteristic.
\label{lem:collisions}
\end{lemma}

\begin{proof}
Due to \eqref{eq:motion_shock_particle}, the difference in function values between the two shocks evolves according to
\begin{align}
\tfrac{d}{dt}\prn{u_{i+1}^- -u_i^+} &=
\prn{\tfrac{s(u_{i+1}^-,u_{i+1}^+)-f'(u_{i+1}^-)}{f''(u_{i+1}^-)}
+\tfrac{f'(u_i^+)-s(u_i^-,u_i^+)}{f''(u_i^+)}}
\prn{f'(u_{i+1}^-)-f'(u_i^+)}\tfrac{1}{x_{i+1}-x_i}\;\notag\\
&=\prn{\tfrac{s(u_{i+1}^-,u_{i+1}^+)-f'(u_{i+1}^-)}{f''(u_{i+1}^-)}
+\tfrac{f'(u_i^+)-s(u_i^-,u_i^+)}{f''(u_i^+)}}
f''(\xi)\frac{u^-_{i+1}-u^+_{i}}{x_{i+1}-x_i}\;.
\end{align}
Here $\xi$ is a value between $u^-_{i+1}-u^+_{i}$ given by the Mean
Value Theorem.
Due to the Oleinik entropy condition, both numerators inside the large
parentheses are non-positive, and by assumption at least one is
strictly negative.
The signs of $f''$ inside and outside the large parentheses cancel
each other out. 
Hence the right hand side always has the opposite sign than $u_{i+1}^-
-u_i^+$. 
If we assume, by negation, that $u_{i+1}^- -u_i^+$ remains finite as
$x_{i+1}-x_i\to 0$ we get a clear contradiction, and thus the difference $u_{i+1}^- -u_i^+$ goes to zero.
\end{proof}

In the computational method, two shock particles $(x_i,u_i^-,u_i^+)$ and $(x_{i+1},u_{i+1}^-,u_{i+1}^+)$
that collide $x_i = x_{i+1}$ are simply merged into a single particle $(x_i,u_i^-,u_{i+1}^+)$. Due to Lemma~\ref{lem:collisions}, this merge does not change the actual solution. Hence, Thm.~\ref{thm:exact_solution} extends to allow particle merges, assuming that the time evolution \eqref{eq:motion_shock_particle} is integrated exactly. If both interacting particles are characteristic, this approach automatically creates a shock.

\section{An ``Exact'' ODE Based Method}
\label{sec:exact_particle_method}
Due to Thm.~\ref{thm:exact_solution}, the presented approach yields the exact weak entropy solution of the conservation law \eqref{eq:conservation_law}, when starting with an initial condition $u_0$ that can be represented by finitely many particles $P_0$, i.e.~$u_0 = U_{P_0}$. Hence, we call this approach an \emph{exact particle method}. In practice, two types of approximation are performed. First, a general initial function $u_0$ cannot be represented exactly using finitely many particles, and thus needs to be approximated. This aspect is briefly addressed in Sect.~\ref{subsec:approximation_initial_conditions}. Second, in general the time evolution \eqref{eq:motion_shock_particle} can not be integrated exactly. Instead, a numerical ODE solver has to be used. This aspect is addressed in Sect.~\ref{subsec:time_integration}.

\subsection{Approximation of the Initial Conditions}
\label{subsec:approximation_initial_conditions}
Whenever the initial function $u_0$ can be represented exactly by an interpolation $U_{P_0}$, one should do so if the number of particles required is computationally acceptable. A particular advantage of the presented particle approach is that discontinuities can be represented exactly. If the initial function cannot be represented exactly, it must be approximated. It is shown in \cite{FarjounSeibold2009_2} that the interpolation \eqref{eq:interpolation} approximates piecewise smooth initial conditions with an error of $O(h^2)$, where $h$ is the maximum distance between particles, if discontinuities are represented exactly. Furthermore, since the method does not require an equidistant placement of particles, adaptive sampling strategies should be used, such as presented in \cite{FarjounSeibold2009_2}. These results are based on the particles be placed exactly on the function $u_0$. More general approximation strategies that do not have this restriction are the subject of current research.

\subsection{Integration in Time}
\label{subsec:time_integration}
The characteristic equation \eqref{eq:characteristic_equations} can easily be integrated exactly. Therefore, characteristic particle movement incurs no integration error, and the next collision time between characteristic particles is explicitly given by \eqref{eq:collision_time}. The particle approach presented and analyzed in \cite{FarjounSeibold2009_1,FarjounSeibold2009_2,FarjounSeibold2009_3} relies on these properties. The downside of those methods is an intrinsic error near shocks. In contrast, the shock-particle method presented in the current paper does not incur any errors around shocks. The downside is an error due to the integration of the ordinary differential equation \eqref{eq:characteristic_equations}. However, it is comparably simple to integrate systems of ODE with very high accuracy. In contrast, the construction of high order numerical approaches that approximate the PDE \eqref{eq:conservation_law} directly (such as ENO~\cite{HartenEngquistOsherChakravarthy1987}/WENO~\cite{LiuOsherChan1994} or Godunov schemes with limiters \cite{VanLeer1974}), is much more challenging. The numerical error analysis shown in Sect.~\ref{sec:numerical_error_analysis_particle_method} seconds this.

Another feature (besides accuracy) that the used numerical ODE solver needs to possess is \emph{event detection}. Since at particle collisions, the system undergoes a discontinuous change (the number of particles is reduced), the ODE solver must detect such events with high accuracy. One way to do this is to use a solver that can provide a high order interpolation. Several such solvers have been derived by Dormand and Prince in \cite{DormandPrince1986} for the Runge-Kutta family of ODE solvers. In {\sc Matlab}, event detection is implemented in particular in \texttt{ode23.m} and \texttt{ode45.m}. As stated in \cite{ShampineReichlet1997}, the latter contains an unpublished\footnote{Details can be found in the {\sc Matlab} file \texttt{ntrp45.m}.} variation of the interpolation presented in \cite{DormandPrince1986}.

To enhance the performance of the adaptive ODE solver it is helpful to have an estimate about the next occurrence of a particle collision. A simple estimate may be obtained by using only the first equation of \eqref{eq:motion_shock_particle}, thus estimating the collision time between neighboring particles by $T_i \approx -\frac{x_{i+1}-x_i}{s(u_{i+1}^-,u_{i+1}^+)-s(u_i^-,u_i^+)}$.

\section{Numerical Error Analysis of the Particle Method}
\label{sec:numerical_error_analysis_particle_method}
We investigate the order of accuracy of the presented particle method, and compare it to the benchmark PDE solver CLAWPACK \cite{Clawpack,LeVeque2002}. The comparison between a finite volume method and a particle method is tricky, since the two approaches are fundamentally different. First, the finite volume approach works with average function values in fixed cells, while with particles the interpolation \eqref{eq:interpolation_shock} defines a function everywhere. This difference can be overcome by constructing errors in the $L^1$ sense from cell averages, as described in \cite{FarjounSeibold2009_2}. Second, the finite volume method has a fixed spacial resolution $\Delta x$, while particles move, merge, and are generally anything but equidistant. Third, in a convergence analysis of a finite volume method, the spacial resolution and the time step are chosen proportional $\Delta t = C\Delta x$. In contrast, the particle method becomes exact if $\Delta t\to 0$, assuming that the initial conditions can be represented exactly by finitely many particles.

\begin{figure}
\begin{minipage}[t]{.24\textwidth}
\includegraphics[width=\textwidth]{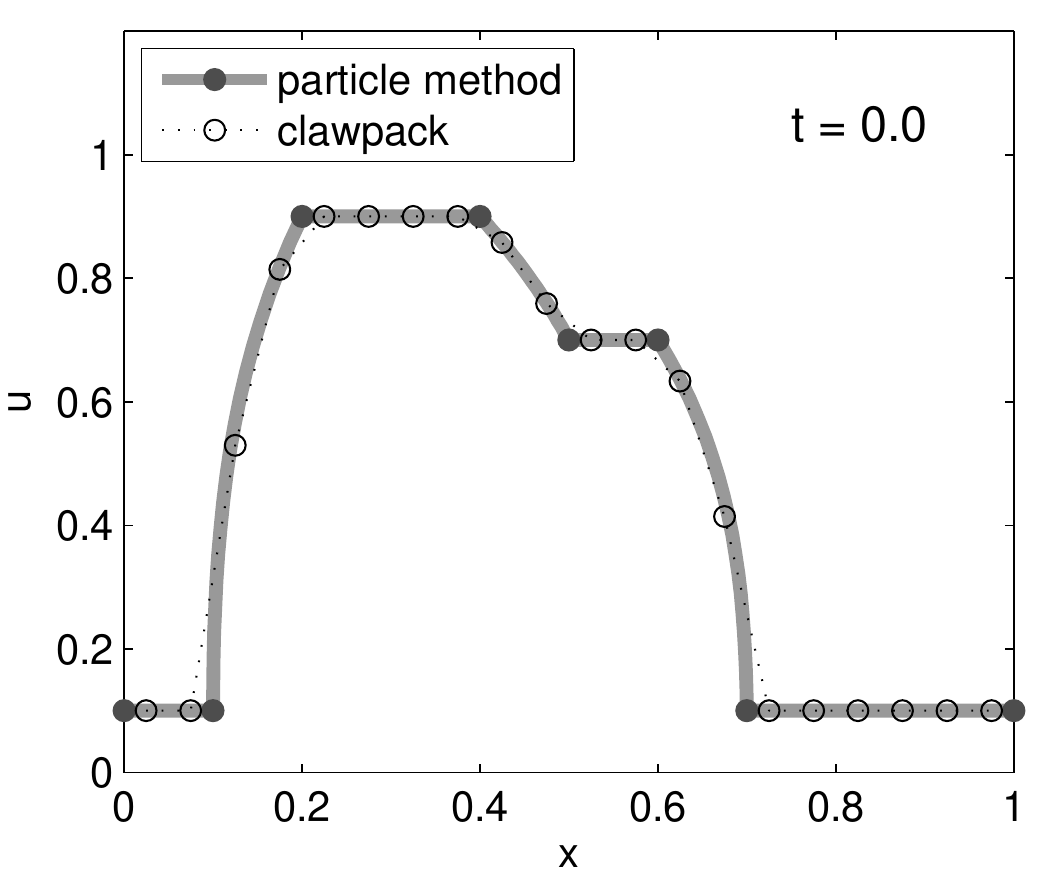}
\end{minipage}
\hfill
\begin{minipage}[t]{.24\textwidth}
\includegraphics[width=\textwidth]{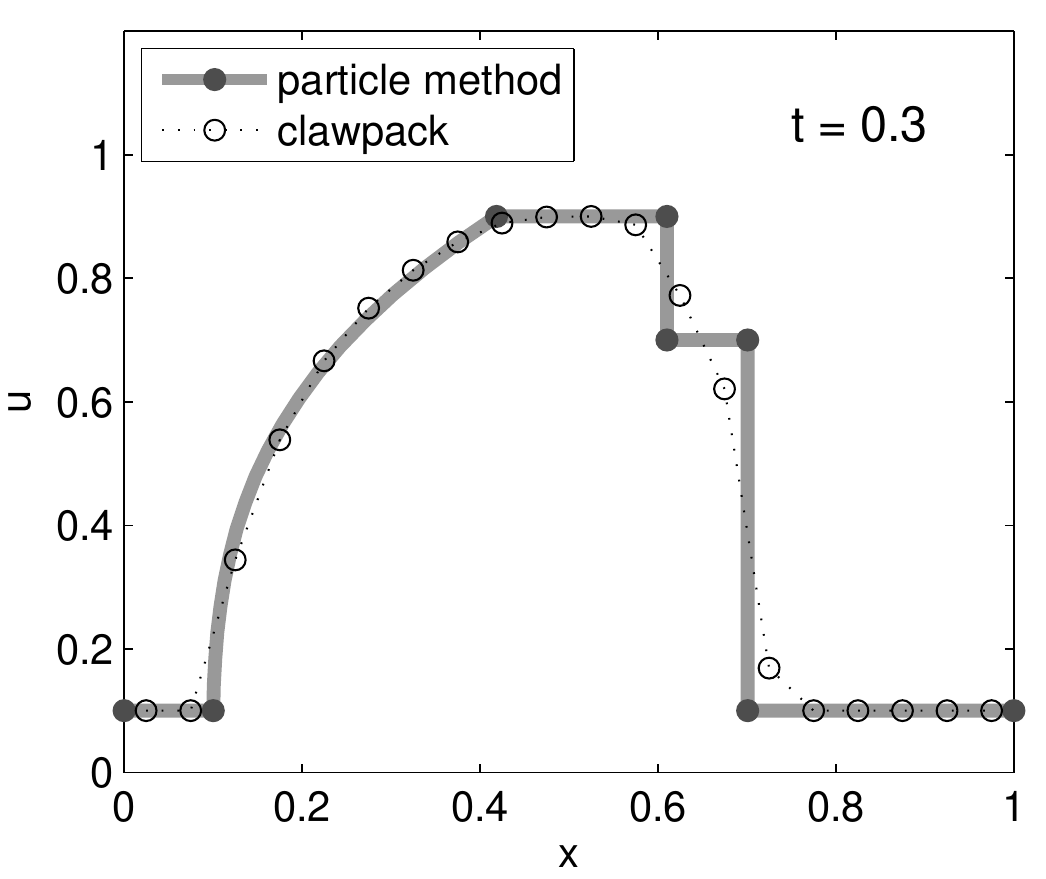}
\end{minipage}
\hfill
\begin{minipage}[t]{.24\textwidth}
\includegraphics[width=\textwidth]{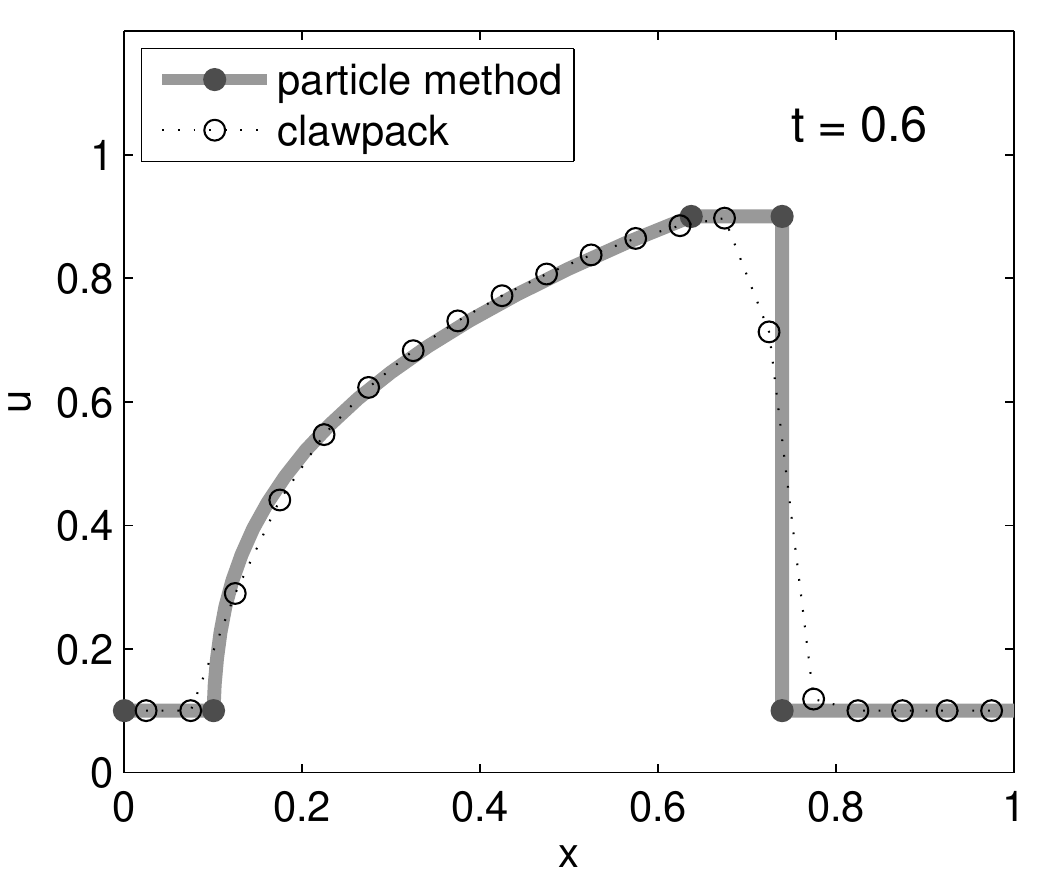}
\end{minipage}
\hfill
\begin{minipage}[t]{.24\textwidth}
\includegraphics[width=\textwidth]{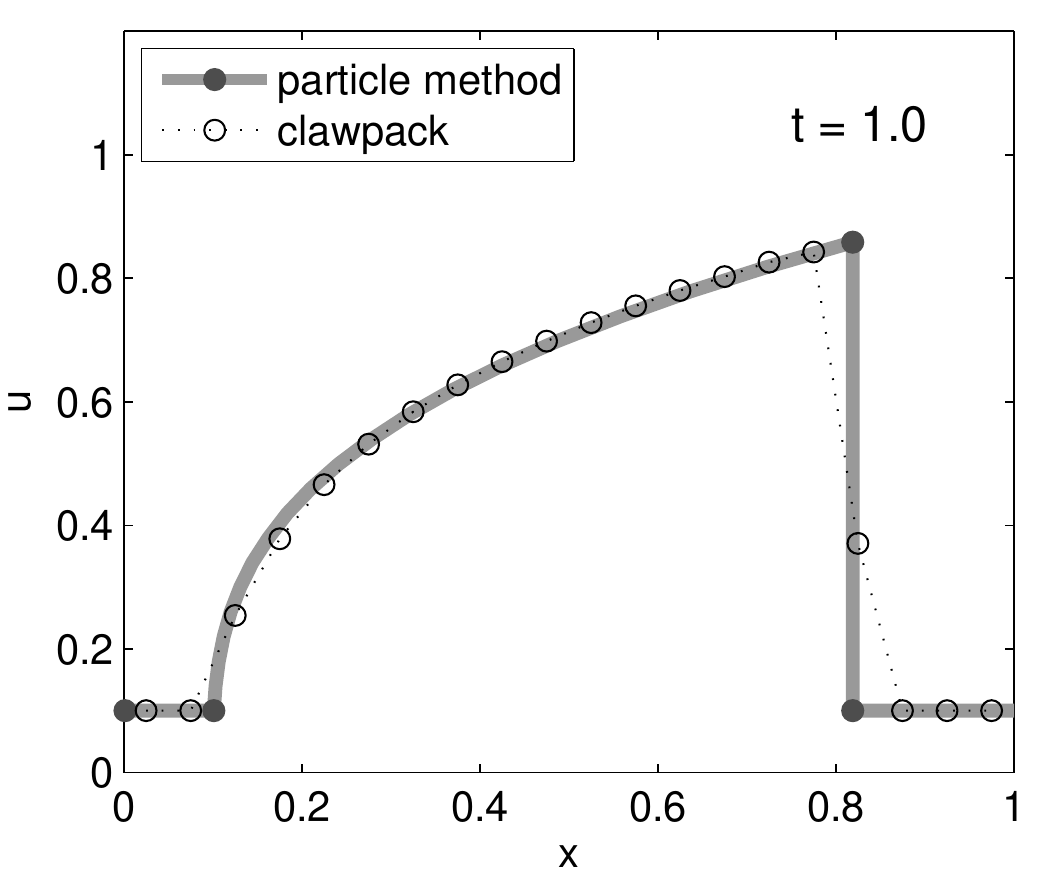}
\end{minipage}
\caption{Time evolution at $t\in\{0,0.3,0.6,1\}$ of the solution to the conservation law \eqref{eq:conservation_law} with $f(u) = \frac{1}{4}u^4$, both by CLAWPACK and by our particle method.}
\label{fig:error_analysis_evolution}
\end{figure}

Here, we consider an initial condition that can be represented exactly by the interpolation \eqref{eq:interpolation_shock}. The reason is that we do not want to measure the error in approximating general initial conditions (for this aspect, please consult \cite{FarjounSeibold2009_2}). Instead, we want to investigate the error in the particle evolution. We consider a second order and a fourth order accurate Runge-Kutta method for the time evolution of \eqref{eq:motion_shock_particle}. Times of particle collisions are found and resolved with the same order of accuracy. For the CLAWPACK runs, we specify a desired CFL number \cite{CourantFriedrichsLewy1928} of 0.8 and let the code choose $\Delta t$ as it finds suitable. In practice, for this problem, this amounts to having $\Delta t\approx \Delta x$.

Specifically, we consider the conservation law \eqref{eq:conservation_law} with flux function $f(u) = \frac{1}{4}u^4$, and initial function $u_0(x) = U_{P_0}(x)$, which is the interpolation \eqref{eq:interpolation} defined by the characteristic particles $P_0 = \{(0,0.1), (0.1,0.1), (0.2,0.9), (0.4,0.9), (0.5,0.7), (0.6,0.7), (0.7,0.1), (1.0,0.1)\}$. The time evolution of the solution is shown in Fig.~\ref{fig:error_analysis_evolution} in four snapshots at $t\in\{0,0.3,0.6,1\}$. In fact, what is shown is the solution obtained by the particle method, integrated with an accuracy that the error is not noticeably in the eye norm. For a comparison, we show the results obtained by a second order CLAWPACK method, with $\Delta x = 0.05$.

\begin{figure}
\centering
\begin{minipage}[t]{.8\textwidth}
\includegraphics[width=\textwidth]{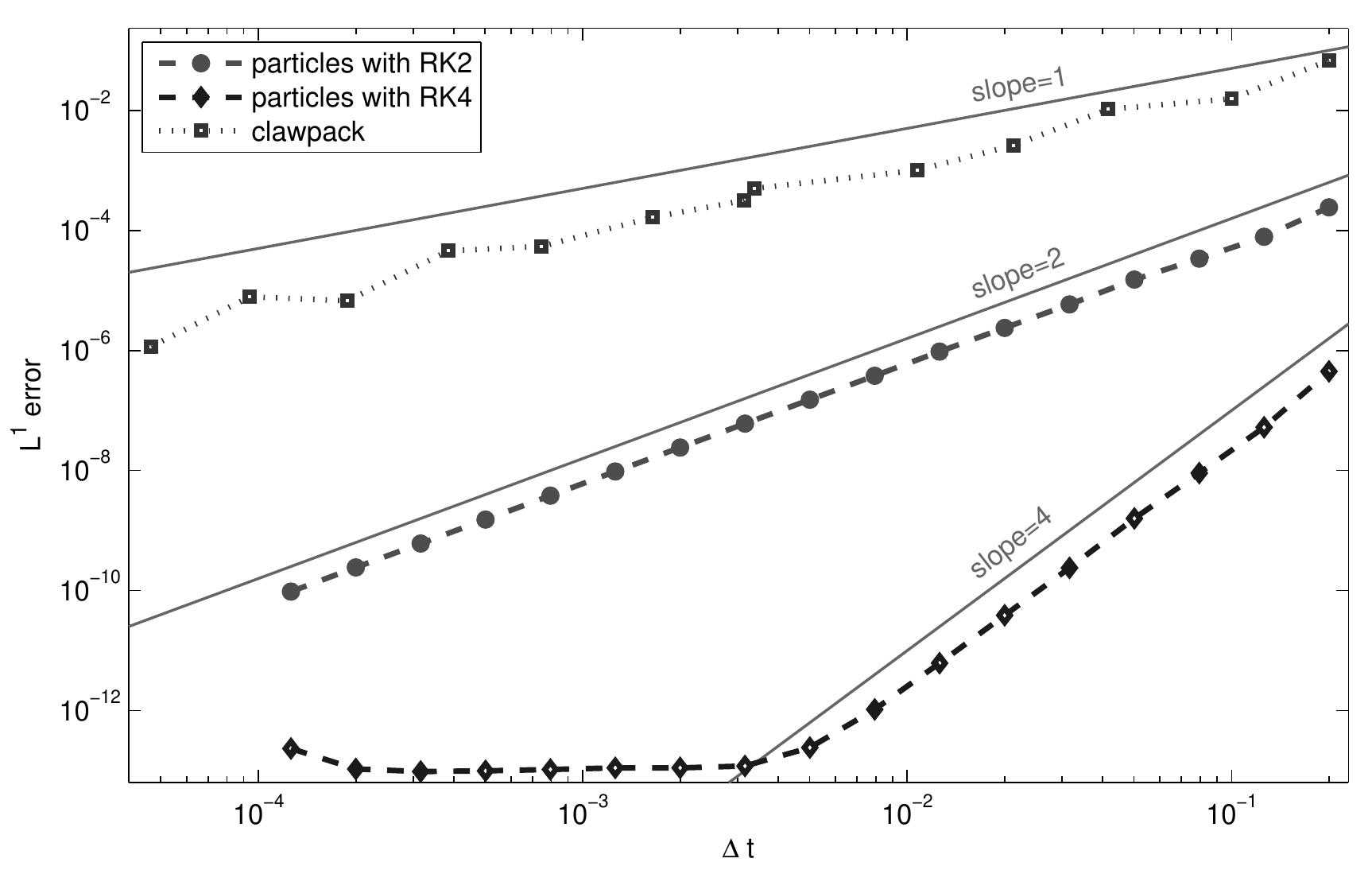}
\end{minipage}
\caption{Error convergence of the particle method in comparison with CLAWPACK. The dashed graphs denote the particle method, with RK2 (dots) and RK4 (diamonds) time stepping. The dotted graph represents CLAWPACK.}
\label{fig:error_analysis}
\end{figure}

The convergence of the error for various approaches is shown in
Fig.~\ref{fig:error_analysis}. We consider the $L^1([0, 1])$ error at
time $t=1$. One can observe that the overall order of the particle
method equals the order of the ODE solver used. Thus, with the
standard RK4 solver, machine accuracy is obtained already for moderate
time steps. Note again that we do not need to increase the number of
particles to obtained convergence. This is a crucial advantage of the
presented particle method. 
In comparison, the CLAWPACK method yields an order of convergence
slightly better than first order. This is particularly due to the
presence of the shocks in the solution, and the large
derivatives in the initial condition.

\section{Stiff Reaction Kinetics}
\label{sec:reaction_kinetics}
Many problems in chemical reaction kinetics can be described by advection-reaction equations, for which the reaction happens on a much faster time scale than the advection. We consider the balance law
\begin{equation}
u_t+\prn{f(u)}_x = \psi(u)\;,
\label{eq:advection_reaction}
\end{equation}
where $0\le u\le 1$ represents the density of some chemical quantity. The advection is given by the nonlinear flux function $f$, which is assumed convex (the case of $f$ concave is analogous) and to be of order $O(1)$. The reactions are described by the source $\psi$. Here, we consider a stiff bistable reaction term
\begin{equation}
\psi(u) = \tfrac{1}{\tau} u(1-u)(u-\beta)
\label{eq:reaction_term}
\end{equation}
where $0<\beta<1$ is a fixed constant. This source term drives the values of $u<\beta$ towards 0, and values of $u>\beta$ towards 1. The reactions happen on a much faster time scale $O(\tau)$, where $\tau\ll 1$. This example is presented for instance in \cite{HelzelLevequeWarnecke2000}. Since the source term $\psi$ does not act in a discontinuity, equation \eqref{eq:advection_reaction} possesses shock solutions as the homogeneous problem \eqref{eq:conservation_law} does. In addition, it has traveling wave solutions that connect a left state $u_L \approx 0$ with a right state $u_R \approx 1$ by a continuous function. To find those solutions, we make a traveling wave ansatz $u(x,t) = v(\xi)$, where $\xi = \frac{x-rt}{\tau}$ is the self similar variable. This transforms equation \eqref{eq:advection_reaction} into a first order ordinary differential equation for $v$
\begin{equation}
v'(\xi) = \tfrac{v(\xi)(1-v(\xi))(v(\xi)-\beta)}{f'(v(\xi))-r}\;.
\label{eq:advection_reaction_traveling_wave}
\end{equation}
For $v = \beta$, the numerator of \eqref{eq:advection_reaction_traveling_wave} vanishes. A solution that connects a state $v_L<\beta$ to a state $v_R>\beta$ can only pass through $v=\beta$ if the denominator of \eqref{eq:advection_reaction_traveling_wave} vanishes as well, i.e.~$r = f'(\beta)$, which yields the velocity of the traveling wave. The shape of the wave $v(\xi)$ is then found by integrating \eqref{eq:advection_reaction_traveling_wave} using $r = f'(\beta)$. Since $u(x,t) = v(\frac{x-rt}{\tau})$, the traveling wave has a thickness (in the $x$-coordinate) of $O(\tau)$. This analysis of traveling wave solutions is in spirit similar to detonation waves of reacting gas dynamics \cite{FickettDavis1979}. The value $\beta$ plays the role of a \emph{sonic point} in detonation waves. The advection-reaction equation \eqref{eq:advection_reaction} is studied in \cite{FanJinTeng2000}. The traveling wave solution \eqref{eq:advection_reaction_traveling_wave} results from a balance of the advection term (which flattens the profile) and the reaction term (which sharpens the profile). Since $\tau$ is very small, these traveling waves look very similar to shocks, yet they face the opposite direction, and travel at a different velocity (if $f'(\beta)\neq f(1)-f(0)$).

In computations, the recovery of the exact shape of the traveling waves is typically not very important. However, the recovery of their correct propagation velocity is crucial. As described in \cite{LeVeque2002}, equation \eqref{eq:advection_reaction} can be treated in a straightforward fashion using classical finite volume approaches. However, correct propagation velocities of traveling waves are only obtained if these are numerically resolved. Thus, with equidistant grids, one is forced to use a very fine grid resolution $h = O(\tau)$, which is unnecessarily costly away from the traveling wave. This problem can be circumvented using adaptive mesh refinement techniques, however, at the expense of simplicity. An alternative approach, presented in \cite{HelzelLevequeWarnecke2000}, yields correct traveling wave velocities even with grid resolutions $h\gg O(\tau)$, by encoding specific information about the structure of the reaction term into a Riemann solver.

\section{A Particle Method for Stiff Reaction Kinetics}
\label{sec:particl_method_stiff_reaction_kinetics}
Here, we present an approach based on the particle method introduced in Sect.~\ref{sec:exact_particle_method} that uses the ``subgrid'' information provided by the interpolation \eqref{eq:interpolation_shock} to yield correct propagation velocities of traveling waves, without specifically resolving them. The characteristic equations for \eqref{eq:advection_reaction} are
\begin{equation}
\left\{
\begin{aligned}
\dot x &= f'(u) \\
\dot u &= \psi(u)\;.
\end{aligned}
\right.
\label{eq:advection_reaction_characteristic_equations}
\end{equation}
As before, our goal is to generalize these characteristic equations to obtain an evolution for shock particles. This requires the definition of an interpolation. We use the interpolant \eqref{eq:interpolation_shock}, as if there were no reaction term. Clearly, this is an approximation, and the resulting method is not exact anymore. At any time $t$, we define the solution by shock particles $P(t) = \{(x_1,u_1^-,u_1^+),\dots,(x_n,u_n^-,u_n^+)\}$, and the interpolation $U_{P(t)}$, defined by \eqref{eq:interpolation_shock}. Adding the reaction term to \eqref{eq:motion_shock_particle}, we now let the particles move according to
\begin{equation}
\left\{
\begin{aligned}
\dot x_i   &= s(u_i^-,u_i^+) \\
\dot u_i^- &= \prn{s(u_i^-,u_i^+)-f'(u_i^-)}
\frac{f'(u_{i-1}^+)-f'(u_i^-)}{x_{i-1}-x_i}\frac{1}{f''(u_i^-)}+\psi(u_i^-) \\
\dot u_i^+ &= \prn{s(u_i^-,u_i^+)-f'(u_i^+)}
\frac{f'(u_{i+1}^-)-f'(u_i^+)}{x_{i+1}-x_i}\frac{1}{f''(u_i^+)}+\psi(u_i^+)
\end{aligned}
\right.
\label{eq:advection_reaction_motion_shock_particle}
\end{equation}
where the shock speed $s(u_i^-,u_i^+)$ is defined as before by \eqref{eq:difference_quotient}. By construction, shocks move at their correct velocity, and for a characteristic particle $u_i^- = u_i^+$, \eqref{eq:advection_reaction_motion_shock_particle} reduces to the correct characteristic evolution \eqref{eq:advection_reaction_characteristic_equations}. Clearly, this approach does not remove the stiffness in time. Hence, an implicit ODE solver should be used. System \eqref{eq:advection_reaction_motion_shock_particle} yields an accurate solution on the particles themselves, as well as an accurate evolution of shocks. However, traveling waves, as given by \eqref{eq:advection_reaction_traveling_wave} are not represented well. The reason is that each particle moves very quickly towards 0 or 1. Then, the reaction term is not considered anymore, since $\psi(0) = 0$ and $\psi(1) = 0$. In order to correctly represent traveling waves, the continuous solution that goes through the sonic point $\beta$ has to be considered. We do so by the following correction approach.

\begin{figure}
\centering
\begin{minipage}[t]{.7\textwidth}
\includegraphics[width=\textwidth]{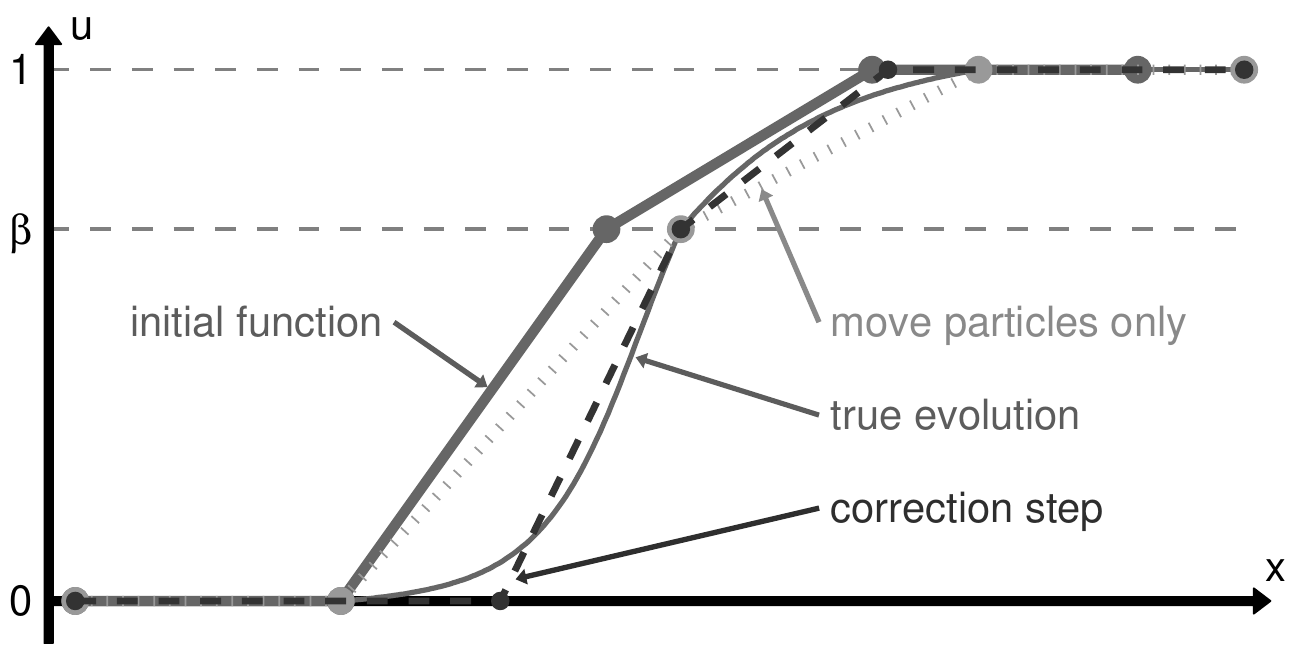}
\end{minipage}
\caption{Correction approach for the advection-reaction equation with dominant reaction term. The vertical dashed lines denote the three roots of the source term.}
\label{fig:advection_reaction_correction}
\end{figure}

We assume that the bistable nature of the reaction term and the value of the unstable root $\beta$ are known. Whenever the solution increases (in $x$) from a value $u<\beta$ to a value $u>\beta$, a special characteristic particle is placed at $u = \beta$, that moves with velocity $\dot{x} = f'(\beta)$. We call such a particle a \emph{sonic particle}. Each particle that neighbors a sonic particle is treated in a special way. As a motivation, consider the situation shown in Fig.~\ref{fig:advection_reaction_correction}: A left state 0 is followed by a sonic particle, which is followed by a right state 1. The interpolant shown is \eqref{eq:interpolation_shock} for $f(u) = \frac{1}{2}u^2$. The thick solid graph shows the initial configuration at time $t$. The dotted graph shows the obtained solution when evolving the particles for a time step $\Delta t$ according to the method of characteristics \eqref{eq:advection_reaction_characteristic_equations}. Since $\psi$ vanishes on all particles, the source is neglected. Hence, this approach does not lead to correct traveling wave solutions. The thin solid graph shows the correct evolution of the initial function, considering the interpolation. This function cannot be represented exactly by particles and the interpolant \eqref{eq:interpolation_shock}, but it can be approximated by modifying the $x$-values of the two particles that neighbor the sonic particle, in such a way that the areas under the solution both left and right of the sonic particle are reproduced correctly. The dashed graph shows the function that results from this correction. Below, we describe the correction approach in detail.

\subsection{Computational Approach}
Consider a particle $(x_{i-1},u_{i-1}^-,u_{i-1}^+)$ that is a left neighbor of a sonic particle $(x_i,\beta)$. The case of a right neighbor particle is analogous. Let the interpolant on $[x_{i-1},x_i]$ be denoted by $U(x)$, and its inverse function $X(u)$. From \eqref{eq:interpolation_shock} it follows that $X'(u) = \frac{x_i-x_{i-1}}{f'(\beta)-f'(u_{i-1}^+)}f''(u)$. Using the interpolant $U(x)$, we can integrate the reaction term between the two particles. The substitution rule yields
\begin{equation}
\int_{x_{i-1}}^{x_i} \!\!\psi(U(x))\ud{x}
= \int_{u_{i-1}^+}^\beta \!\!\psi(u)X'(u)\ud{u}
= \frac{x_i-x_{i-1}}{f'(\beta)-f'(u_{i-1}^+)}
\int_{u_{i-1}^+}^\beta \!\!\psi(u)f''(u)\ud{u}\;.
\label{eq:reaction_term_integrated}
\end{equation}
This expression represents the full influence of the reaction term on the continuous solution between the value $u_{i-1}^+$ and the sonic value $\beta$. As derived in \cite{FarjounSeibold2009_2}, the area under the interpolant on $[x_{i-1},x_i]$ is given by
\begin{equation*}
\int_{x_{i-1}}^{x_i} U(x)\ud{x} = (x_i-x_{i-1})\,a(u_{i-1}^+,\beta)\;,
\end{equation*}
where $a(v,w) = \frac{\brk{f'(u)u-f(u)}_v^w}{\brk{f'(u)}_v^w}$ is a nonlinear average.
Now consider a new particle $(x_{i-1}+\Delta x_{i-1},u_{i-1}^+,u_{i-1}^+)$ be inserted between $x_{i-1}$ and $x_i$. This insertion changes the area by
\begin{equation}
\Delta A = \Delta x_{i-1}\prn{a(u_{i-1}^+,\beta)-u_{i-1}^+}
= \Delta x_{i-1}\frac{f'(\beta)(\beta-u_{i-1}^+)-(f(\beta)-f(u_{i-1}^+))}{f'(\beta)-f'(u_{i-1}^+)}\;.
\label{eq:area_change}
\end{equation}
Equating the rate of area change $\frac{\Delta A}{\Delta t}$, given by \eqref{eq:area_change}, and expression \eqref{eq:reaction_term_integrated}, yields
\begin{equation*}
\tfrac{\Delta x_{i-1}}{\Delta t} = c(u_{i-1}^+,\beta)\prn{x_i-x_{i-1}}\;,
\end{equation*}
where $c(v,w) = \frac{\int_v^w \psi(u)f''(u)\ud{u}}{f'(w)(w-v)-(f(w)-f(v))}$. The scaling $f = O(1)$ and $\psi = O(\tfrac{1}{\tau})$ implies that $c(v,w) = O(\tfrac{1}{\tau})$, if $w-v = O(1)$. A similar derivation for the right neighbor yields that a new particle $(x_{i+1}-\Delta x_{i+1},u_{i+1}^-,u_{i+1}^-)$ needs to be inserted with
\begin{equation*}
\tfrac{\Delta x_{i+1}}{\Delta t} = c(u_{i+1}^-,\beta)\prn{x_i-x_{i+1}}\;.
\end{equation*}
Due to the bistable nature of the reaction term, one will frequently encounter a nearly constant left state, i.e.~both $\abs{u_{i-1}^+-u_{i-1}^-}$ and $\abs{u_{i-2}^+-u_{i-1}^-}$ are very small. In this case, the particle $i-1$ can just be moved by $\Delta x_{i-1}$, instead of creating a new particle. Using a characteristic particle notation only, the resulting modified evolution equations are
\begin{equation*}
\left\{
\begin{aligned}
\dot x_{i-1} &= f'(u_{i-1})+c(u_{i-1},\beta)\prn{x_{i}-x_{i-1}} \\
\dot u_{i-1} &= \psi(u_{i-1}) \\
\dot x_i     &= f'(\beta) \\
\dot u_i     &= 0 \\
\dot x_{i+1} &= f'(u_{i+1})+c(u_{i+1},\beta)\prn{x_{i+1}-x_{i}} \\
\dot u_{i+1} &= \psi(u_{i+1})\;.
\end{aligned}
\right.
\end{equation*}
This implies that
\begin{equation*}
\tfrac{d}{dt}\prn{x_{i}-x_{i-1}} = \prn{f'(\beta)-f'(u_{i-1})}-c(u_{i-1},\beta)\prn{x_{i}-x_{i-1}}\;,
\end{equation*}
i.e.~the distance $x_i-x_{i-1}$ converges to an equilibrium value $\frac{f'(\beta)-f'(u_{i-1})}{c(u_{i-1},\beta)}$. Since $c(u_{i-1},\beta) = O(\tfrac{1}{\tau})$, the equilibrium distance between the sonic particle and its neighbors is $O(\tau)$. Hence, the presented approach yields a traveling wave solution, represented by three particles that move at the correct velocity $f'(\beta)$, and whose distance from each other scales, correctly, like $O(\tau)$.

\section{Numerical Results on Reaction Kinetics}
\label{sec:numerical_results_reaction_kinetics}
For assessing our method numerically, we compare it to the benchmark PDE solver CLAWPACK \cite{Clawpack}, for the advection-reaction equation \eqref{eq:advection_reaction}. We consider the reaction term \eqref{eq:reaction_term} with $\beta = 0.8$, and choose the Burgers' flux $f(u) = \frac{1}{2}u^2$. Four different values for the reaction time scale are considered: $\tau\in\{0.1,0.024,0.008,0.004\}$. The spacial resolution is $\Delta x = 0.02$. As initial condition, we use $u(x,0) = 0.9\,\exp\prn{-150(x-\tfrac{1}{2})^4}$. For solving this problem using CLAWPACK we simply use the code from the CLAWPACK website \cite[Chapter 17]{LeVeque2002}. This code was written specifically to solve this stiff Burgers' problem.

\begin{figure}
\begin{minipage}[t]{.24\textwidth}
\includegraphics[width=\textwidth]{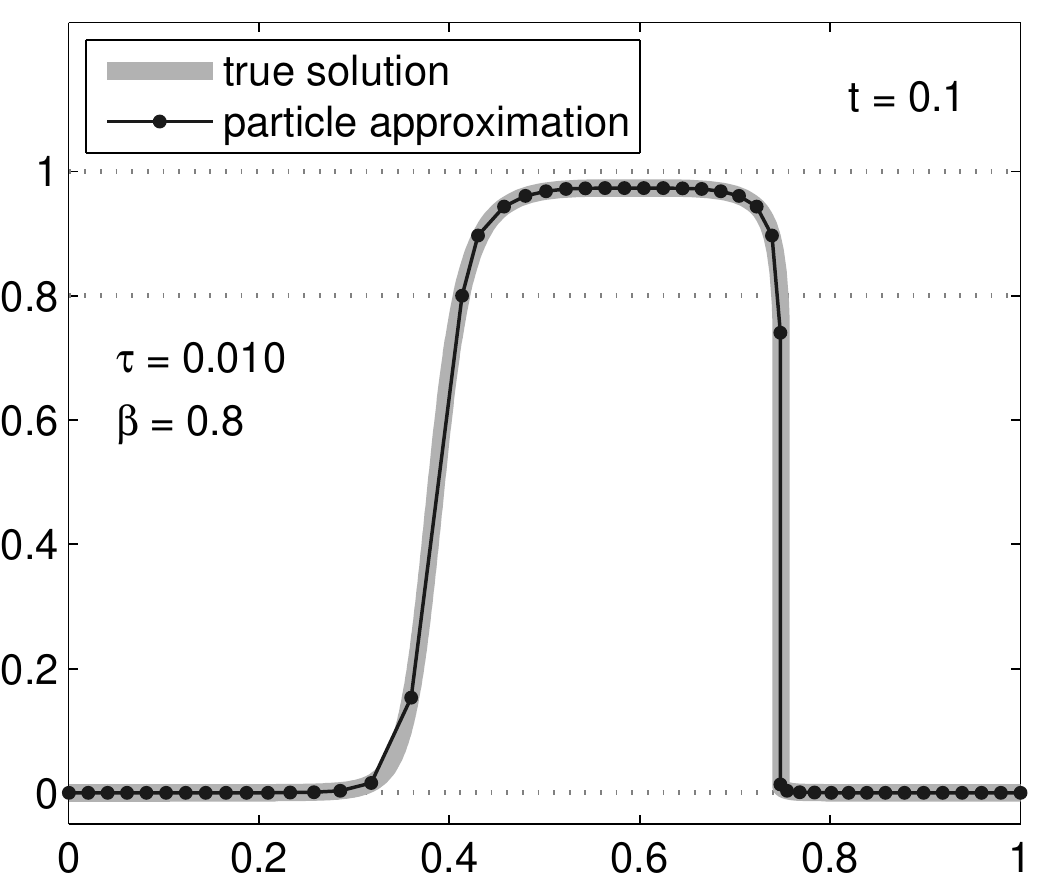}
\end{minipage}
\hfill
\begin{minipage}[t]{.24\textwidth}
\includegraphics[width=\textwidth]{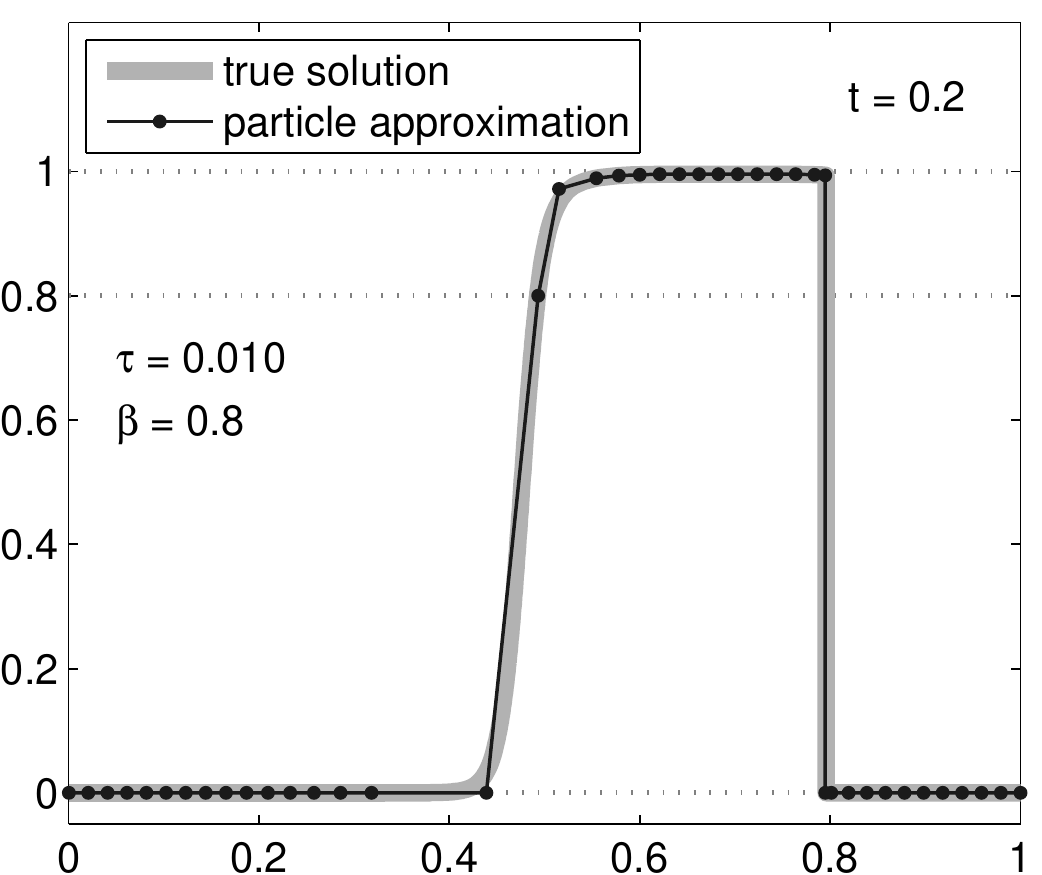}
\end{minipage}
\hfill
\begin{minipage}[t]{.24\textwidth}
\includegraphics[width=\textwidth]{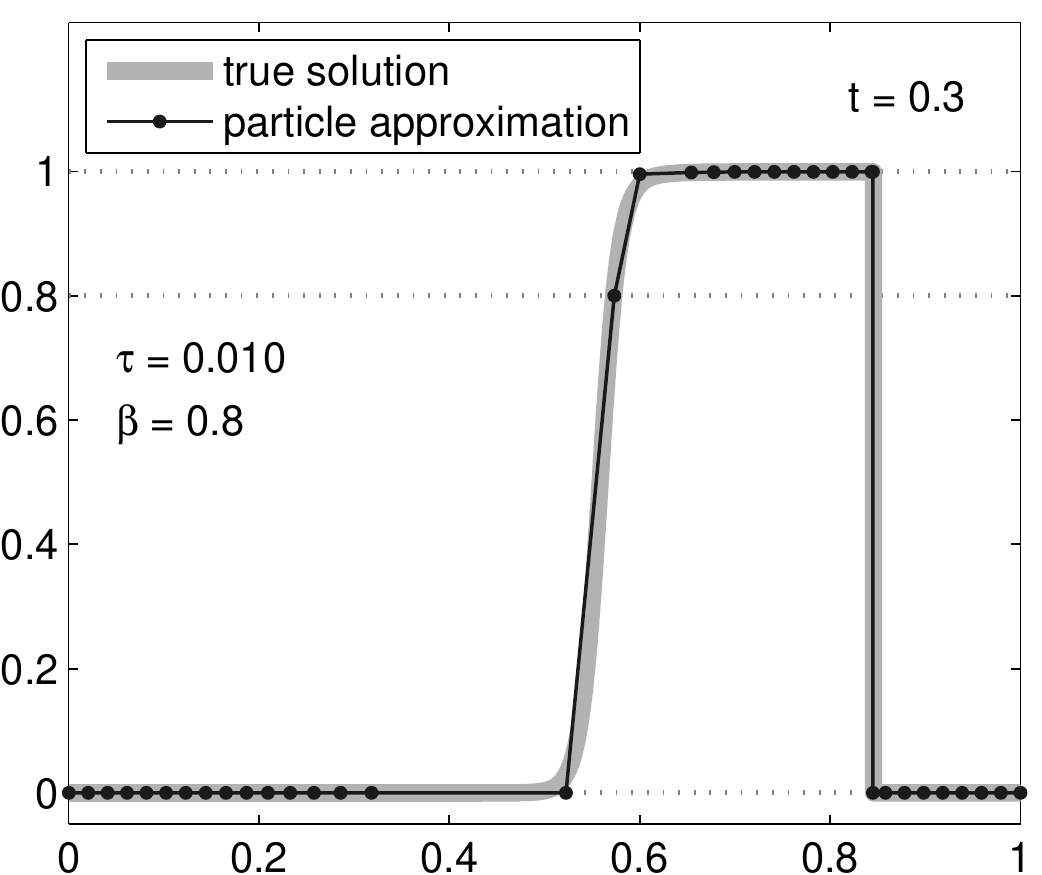}
\end{minipage}
\hfill
\begin{minipage}[t]{.24\textwidth}
\includegraphics[width=\textwidth]{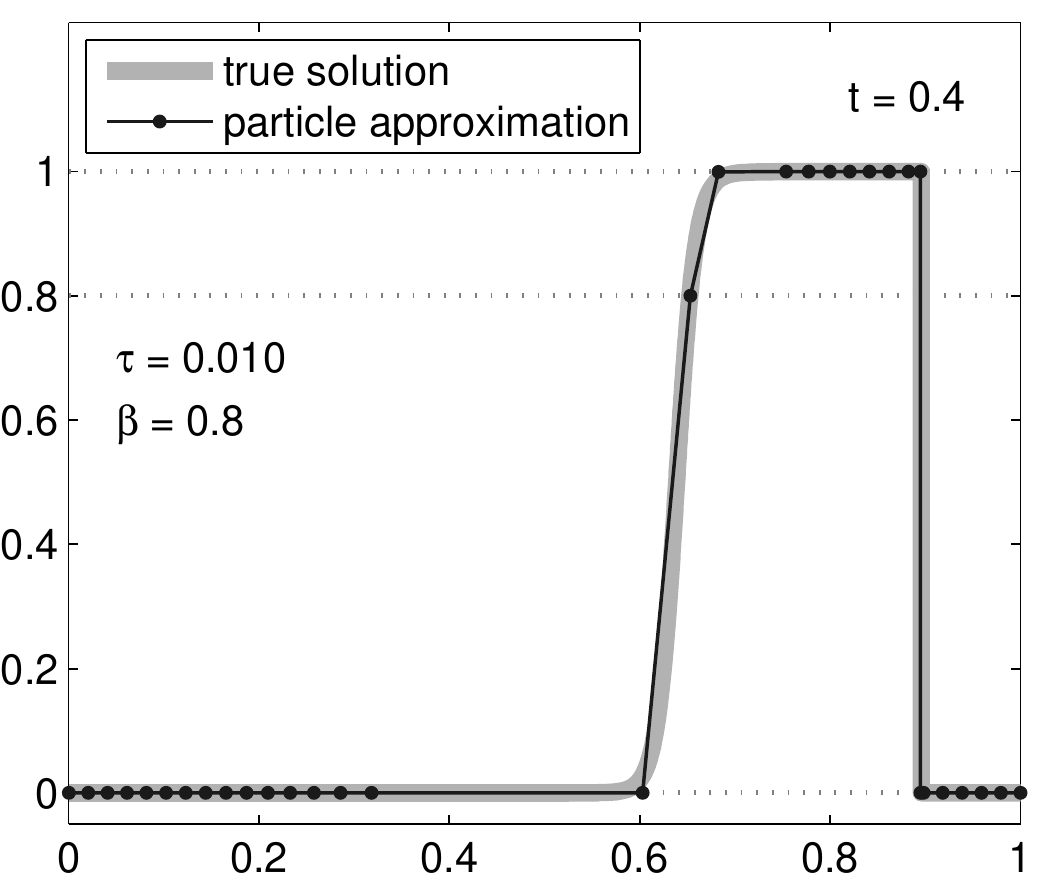}
\end{minipage}
\caption{Time evolution at $t\in\{0.1,0.2,0.3,0.4\}$ of the advection reaction equation \eqref{eq:advection_reaction} with $\tau = 0.01$. The thick gray graph shows the true solution, while the dots denote the particle approximation.}
\label{fig:advection_reaction_evolution}
\end{figure}

The time evolution of the solution of \eqref{eq:advection_reaction} is shown in Fig.~\ref{fig:advection_reaction_evolution} in four snapshots at $t\in\{0.1,0.2,0.3,0.4\}$. The thick grey graph shows the true solution. The solid dots denote the particle method. One can see that at $t=0.1$ the solution is still in the transient phase, since characteristic particles are still visible on the (soon-to-be) detonation wave. At $t=0.2$, the wave structure is almost converged. The plots at $t=0.3$ and $t=0.4$ show how the detonation wave catches up to the shock.

\begin{figure}
\begin{minipage}[t]{.24\textwidth}
\includegraphics[width=\textwidth]{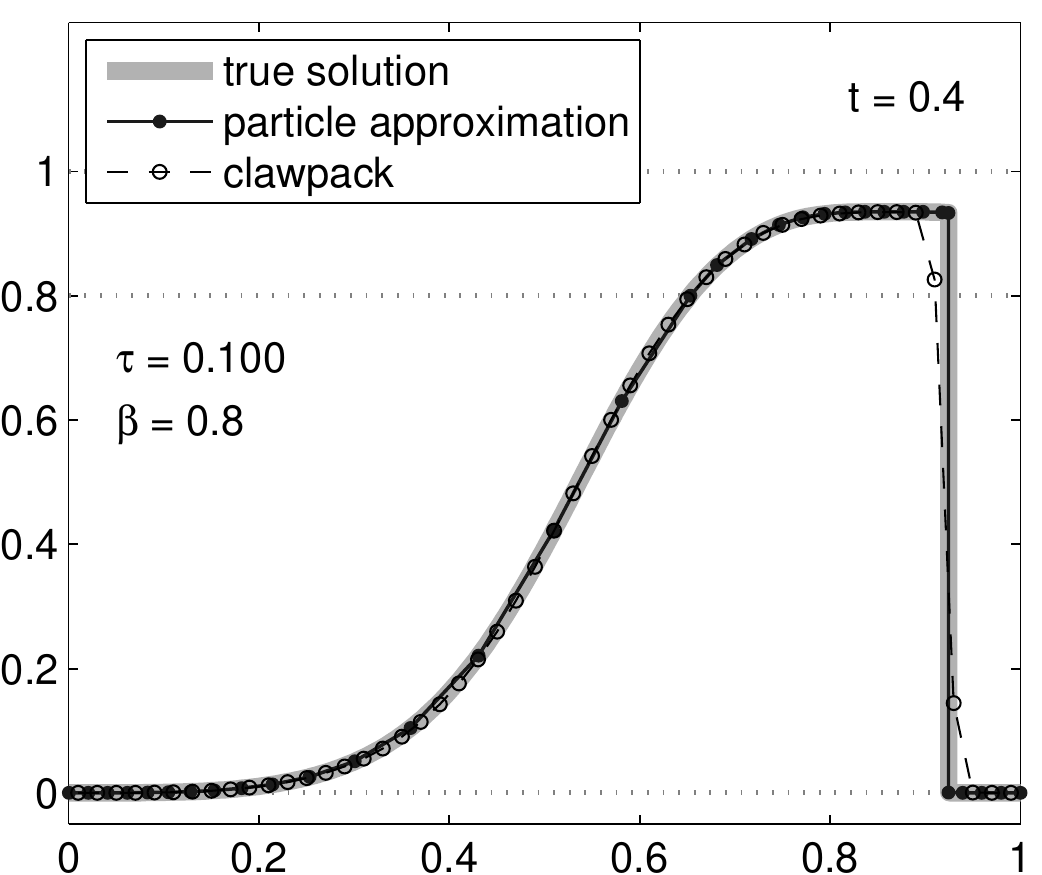}
\end{minipage}
\hfill
\begin{minipage}[t]{.24\textwidth}
\includegraphics[width=\textwidth]{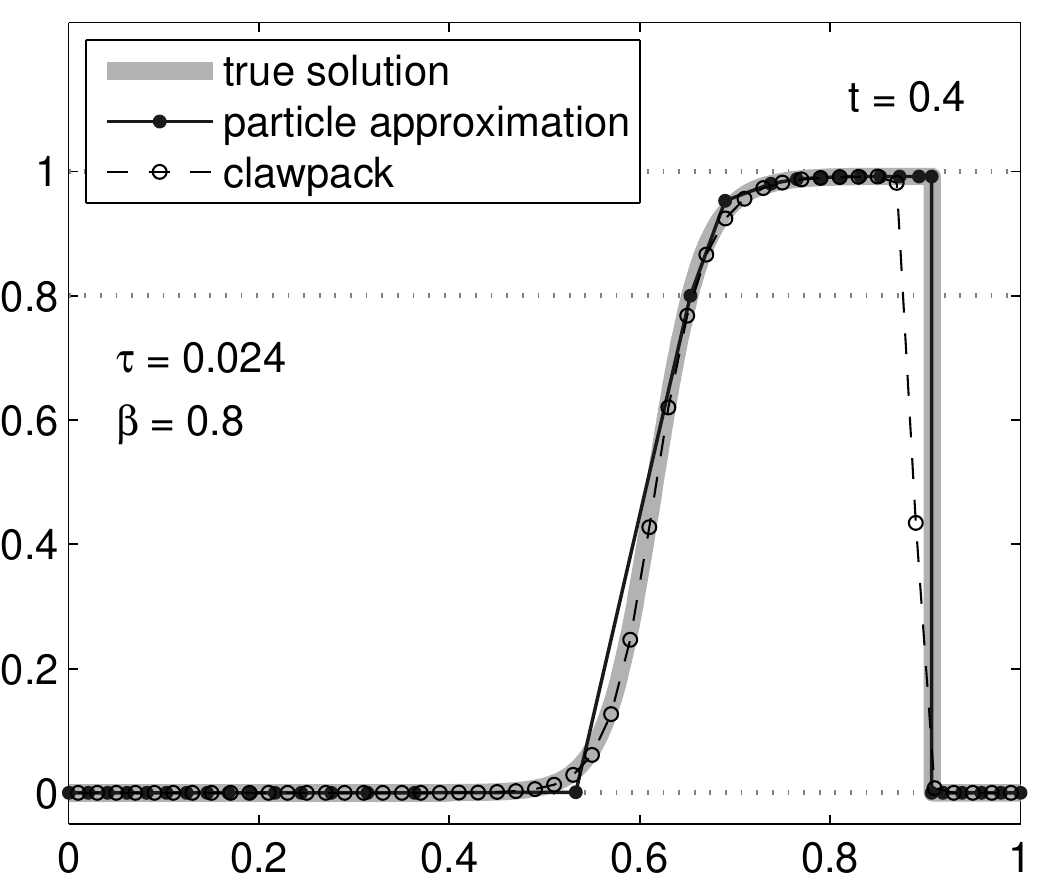}
\end{minipage}
\hfill
\begin{minipage}[t]{.24\textwidth}
\includegraphics[width=\textwidth]{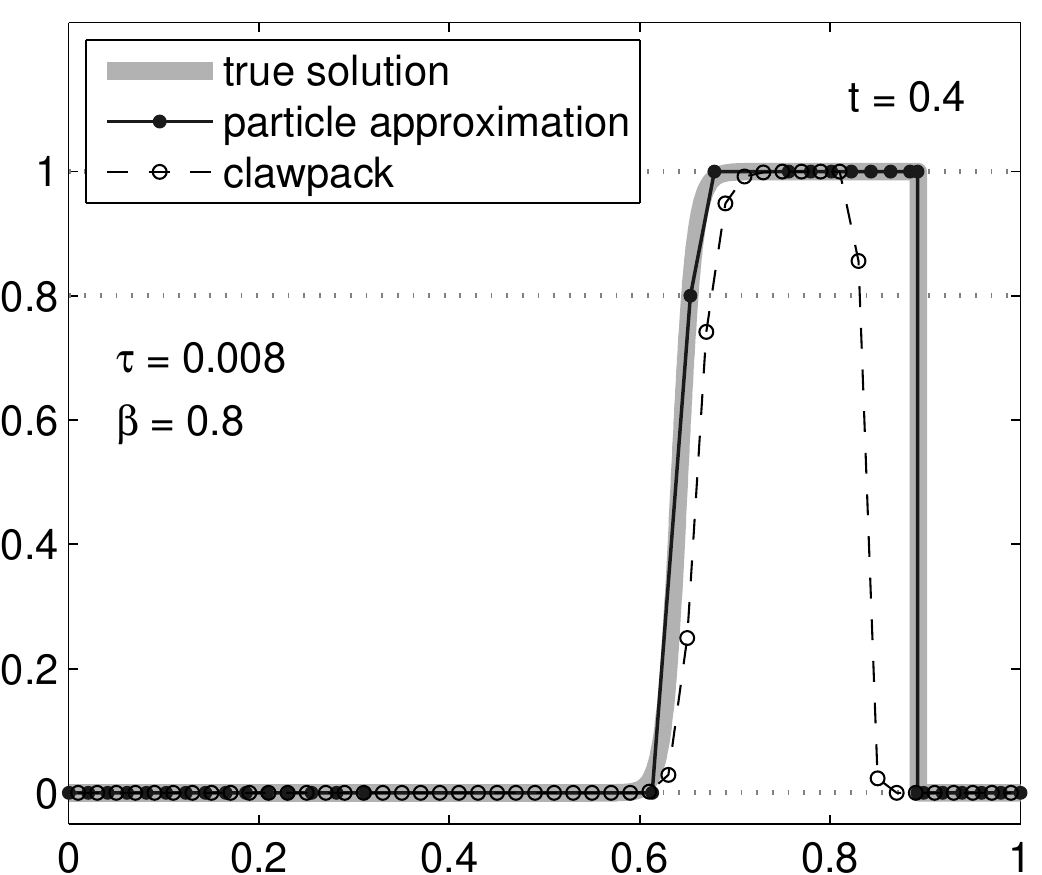}
\end{minipage}
\hfill
\begin{minipage}[t]{.24\textwidth}
\includegraphics[width=\textwidth]{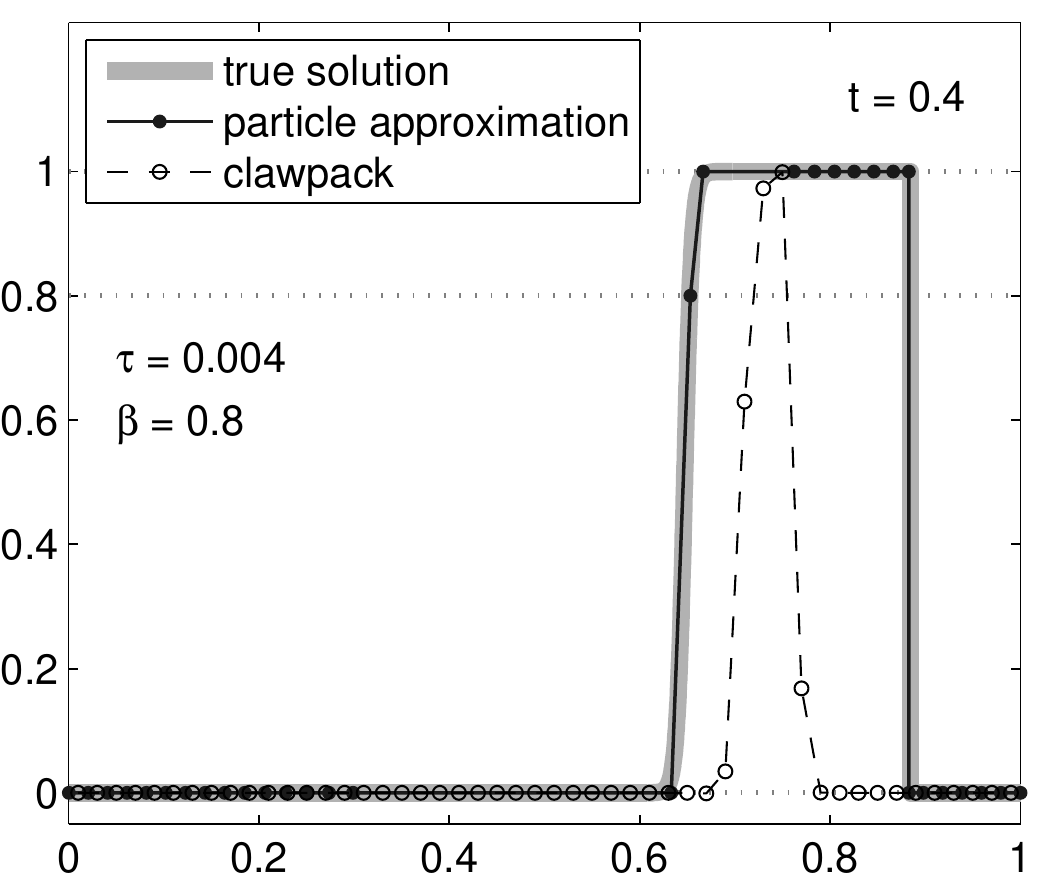}
\end{minipage}
\caption{Computational results for the advection reaction equation \eqref{eq:advection_reaction} with $\tau\in\{0.1,0.024,0.008,0.004\}$. The thick grey graph shows the true solution, while the dots denote the particle approximation.}
\label{fig:advection_reaction_comparison}
\end{figure}

Figure~\ref{fig:advection_reaction_comparison} shows the solution at the final time $t=0.4$, for four choices of $\tau\in\{0.1,0.024,0.008,0.004\}$. The thick grey graph shows the true solution. The solid dots denote the particle method. The circles show the CLAWPACK results. Note that for $\tau=0.1$, the solution is still in the transient phase, while for the other values, the detonation wave is comparably well established. For the selected resolution $\Delta x = 0.02$, CLAWPACK successfully captures the shock for $\tau\in\{0.1,0.024\}$. The detonation wave for $\tau=0.024$ is nicely represented as well. However, CLAWPACK clearly fails to resolve the shock and the detonation for $\tau=0.004$. The intermediate value $\tau=0.008$ is on the edge of failure. In comparison, the particle method works for all values of $\tau$. The shock is optimally sharp, and the detonation wave moves at the correct velocity and has the correct width. The trouble that CLAWPACK is having with these equations is due to the stiff source. The problem is that the width of the shock is always $O(\Delta x)$, but this is too large when $\tau$ becomes small. The source is too active both in the detonation shock and in the regular forward-facing shock. This leads to incorrect shock speeds.

\section{Conclusions and Outlook}
We have presented a particle method that solves scalar one-dimensional hyperbolic conservation laws exactly, up to the accuracy of an ODE solver and up to errors in the approximation of the initial conditions. The numerical solution is defined everywhere. It is composed of local similarity solutions, separated by shocks. A numerical convergence analysis verified this accuracy claim for the flux function $f(u)=u^4/4$. In this example, the basic RK4 method yields solutions up to machine accuracy using a few hundred time steps. Since general initial conditions can be approximated with second-order accuracy (see \cite{FarjounSeibold2009_2}), the overall method is at least second order accurate, even in the presence of shocks.

The method has also been extended to balance laws that describe stiff reaction kinetics. The tracking of a sonic particle in combination with a correction approach for neighboring particles yields a method that evolves detonation waves at correct velocities, without actually resolving their internal dynamics. The evolution of the sonic particle comes naturally in the considered particle method, while for classical fixed grid methods, a similar approach is much less natural. Numerical tests show that the particle method approximates the true solutions very well, even for fairly stiff systems, for which CLAWPACK fails due to an under-resolution of the wave and the shock.

The philosophy of the considered application in stiff reaction kinetics is that one can find efficient approaches for more complex problems by using the exact conservation law solver at the basis. It is the subject of current and future research to apply the same philosophy in other applications. Examples are nonlinear flows on networks. The presented particle method can be used to solve the actual evolution on each edge exactly. While an approximation has to be done at the network nodes, it is plausible to believe that this approach yields more accurate results that classical method that are far from exact on the edges themselves. Further generalizations to consider are the treatment of higher space dimensions using dimensional splitting, and systems of conservation/balance laws.

\section*{Acknowledgments}
The authors would like to acknowledge the support by the
National Science Foundation.
Y.~Farjoun was partially supported by NSF grant DMS--0703937.
B.~Seibold was partially supported by NSF grant DMS--0813648.
The first author was also financed by the Spanish Ministry of Science and Innovation under grant
FIS2008-04921-C02-01.

\bibliographystyle{plain}
\bibliography{references_complete}

\end{document}